\documentclass[letterpaper, 10pt, conference]{ieeeconf} 
\IEEEoverridecommandlockouts 
\usepackage{graphics} 
\usepackage{epsfig} 
\usepackage{times} 
\usepackage{amsmath} 
\usepackage{amssymb}  
\usepackage{amsmath} 
\usepackage{amssymb}  

\usepackage{color}
\usepackage{hyperref}

\newtheorem{thm}{Theorem}

\newtheorem{defn}{Definition}
\newtheorem{lem}{Lemma}
\newtheorem{prop}{Proposition}
\usepackage{algorithm, algorithmicx, algpseudocode}

\newcommand{\R}{{\mathbb R}}

\newcommand{\C}{{\mathbb C}}

\newcommand{\cB}{{\mathcal B}}

\newcommand{\cD}{{\mathcal D}}
\newcommand{\cDDp}{{\mathcal{DD}_+}}

\newcommand{\Hinf}{{\mathbb{H}_\infty}}

\newcommand{\cH}{{\mathcal H}}

\newcommand{\cK}{{\mathcal K}}

\newcommand{\cM}{{\mathcal M}}

\newcommand{\cS}{{\mathcal S}}
\newcommand{\cT}{{\mathcal T}}

\newcommand{\mH}{\mathcal H}
\newcommand{\mHp}{\mathcal H_{+}}

\newcommand{\tr}{\textrm{trace}}

\newcommand{\diag}[1]{\textrm{diag}\{#1\}}

\renewcommand{\Re}{\mathrm{Re}}

\newcommand{\ja}[1]{{\color{red}#1}}
\begin{document}
\title{On Existence of Solutions to Structured Lyapunov Inequalities}
\author{Aivar Sootla and James Anderson \\\\  \ja{A condensed version of this paper will appear} \\ \ja{ in the Proceedings of the 2016 American Control Conference }
\thanks{J. Anderson is with St John's College, Oxford and the Department of Engineering Science, University of Oxford, Parks Road, Oxford, OX1 3PJ, U.K. e-mail: james.anderson@eng.ox.acuk }
\thanks{A. Sootla is with the Montefiore Institute, University of Li\`ege, B28, Li\`ege Belgium, B4000 e-mail: asootla@ulg.ac.be. A. Sootla holds an F.R.S--FNRS fellowship. This paper is partially funded through the Belgian Network DYSCO, and the Interuniversity Attraction Poles Programme initiated by the Belgian Science Policy Office.}
\thanks{The authors would like to thank Prof. Amir Ali Ahmadi for valuable discussions, and specifically for pointing out the reference~\cite{basis_pursuit}.}}

\maketitle
\begin{abstract} 
In this paper, we derive sufficient conditions on drift matrices under which block-diagonal solutions to Lyapunov inequalities exist. The motivation for the problem comes from a recently proposed basis pursuit algorithm. In particular, this algorithm can provide approximate solutions to optimisation programmes with constraints involving Lyapunov inequalities  using linear or second order cone programming. This algorithm requires an initial feasible point, which we aim to provide in this paper. Our existence conditions are based on the so-called $\mH$ matrices. We also establish a link between $\mH$ matrices and an application of a small gain theorem to the drift matrix. We finally show how to construct these solutions in some cases without solving the full Lyapunov inequality. 
\end{abstract}

\section{Introduction}
Lyapunov equations and matrix inequalities play a central role in control theory, since they are used for, e.g., verifying stability of a dynamical systems, optimal control, and model order reduction (cf.~\cite{ZDG}). Lyapunov matrix inequalities with sparsity constraints on the decision variables are used in in the context of distributed control~\cite{linfarjovTAC13admm}, structured model reduction~\cite{Sandberg09} etc. In such applications, a typical constraint on the decision variables is block-diagonality of a matrix. The major bottleneck in solving optimisation programmes with a Lyapunov inequality constraint is scalability, since it is a semidefinite programme (SDP). 
There exist a number of methods addressing scalability of SDPs (cf.~\cite{MasP14,KimKMY11, SDPARA}), and in one of them, it was proposed to replace the constraints in the cone of positive semidefinite matrices with  conic inner-approximations~\cite{dsos_ciss14,MajAT14}. There are two main conic approximations: one which results in a linear programme (LP), and another which results in a second order cone programme (SOCP).  Since we are dealing with inner approximations of the cone of positive semidefinite matrices, even if the  LP or SOCP solution can be computed, this solution is usually conservative with respect to the optimal SDP solution. This limitation was partially addressed using the basis pursuit algorithm~\cite{basis_pursuit}, which is iterates over LPs or SOCPs and provides a guarantee of improvement with each iteration. This algorithm requires an initial feasible point in order to start the iterations. Hence, major questions still remain concerning existence theorems and scalable computation of block-diagonal solutions to Lyapunov inequalities.

Necessary and sufficient conditions for block-diagonal stability were described more than 20 years ago in~\cite{carlson1992block}. However, these results do not provide a constructive way to build block-diagonal Lyapunov functions. This perhaps explains why these results are relatively unused 
in the control theory literature. Besides some simple cases, such as, the drift matrix being block-triangular matrix (cf.~\cite{andersondecentralised}), it is known that the closed loop interconnection of strictly passive systems has a drift matrix which admits a block-diagonal solution to Lyapunov inequalities~\cite{trnka2013structured}. It is also well-known that stable \emph{Metzler} matrices admit diagonal solutions to the Lyapunov inequality~\cite{berman1994nonnegative}. Additional special cases are covered in~\cite{arcak2006diagonal}, \cite{arcak2011diagonal} and revisited in what follows.

In this paper, we aim at identifying additional cases, when a block-diagonal solution to a Lyapunov inequality can be found using algebraic methods or LPs. We start by studying a generalisation of Metzler matrices known as \emph{$\mH$ matrices}. Stable $\mH$ matrices possess many properties of stable Metzler matrices, for example they also admit diagonal solutions to Lyapunov inequalities~\cite{hershkowitz1985lyapunov}. We provide another such property, namely we show that for $\mH$ matrices, diagonal solutions to Lyapunov inequalities can be computed using algebraic methods and/or LPs. We then investigate conditions on specific blocks in block-partitioned matrices. We establish a link between the $\mH$ matrix conditions and a version of the small gain theorem before extending this intuition to block-partitioned case. In the $2$ by $2$ block partitioned case, we provide an explicit way to construct block-diagonal solutions to the Lyapunov inequalities without the need to solve the full inequality. An extension to $n$ by $n$ block partitioned case is one of the future work directions.

 
The rest of the paper is organised as follows. In Section~\ref{s:prel}, we cover some preliminaries and motivate our problem formulation in Section~\ref{s:motivation}. We show how to construct diagonal solutions to Lyapunov inequalities for $\mH$ drift matrices in Section~\ref{s:sdd}. We provide stability results for block partitioned matrices and link the condition for $\mH$ matrices with the small gain theorem in Section~\ref{s:bsdd}. In Section~\ref{sec:example} we provide a large-scale numerical example and we conclude in Section~\ref{s:con}, where we discuss linear programming solutions to Lyapunov inequalities.

\emph{\textbf{Notation}:} Our notation is mostly standard: $\rho(A)$ stands for the spectral radius of a matrix $A$, $A\ge 0$ (respectively, $A\gg 0$) means that all entries $a_{i j}$ of $A$ are nonnegative (respectively, positive), $A\succeq 0$ (respectively, $A\succ 0$) means that $A$ is positive semidefinite (respectively, positive definite). 
Let $\cS^n$ denote the set of symmetric $n$ by $n$ matrices, $\cS^n_+$ denotes the cone of positive semidefinite  $n$ by $n$ matrices. Let $\|B\|_2$ be the matrix induced norm, that is $\|B\|_2$ is equal to the maximum singular value of $B$, and let $\underline\sigma(B)$ denote the minimum singular value of $B$. The $\Hinf$ norm of a transfer matrix $G$ is defined as $\|G\|_{\Hinf} = \max\limits_{s\in\C: \Re(s) \ge 0} \|G(s)\|_2$  and $\|G\|_\Hinf =\max\limits_{\omega \in\R} \|G(\jmath \omega)\|_2$ for stable $G$. For a space $X$, its dual is denoted as $X^\ast$. Finally, $I$ is the identity matrix of an appropriate dimension.
\section{Preliminaries} \label{s:prel}
Consider the linear time invariant dynamical system 
\begin{equation}\label{eq:sys}
\dot{x}(t)=Ax(t), \quad x(0)=x_0
\end{equation}
where $x(t)\in \R^n$. An important concept associated with the system~\eqref{eq:sys} is stability, which is typically verified by solving a linear matrix inequality (LMI).

\begin{prop}
System \eqref{eq:sys} is stable if and only if there exists an $X\succ 0$ that satisfies the LMI 
\begin{equation}\label{eq:lyap_lmi}
A X + X A^T \prec 0.
\end{equation}
\end{prop}

A matrix $X$ which satisfies~\eqref{eq:lyap_lmi} defines a Lyapunov function of the form $V(x)=x(t)^T X^{-1} x(t)$ for system~\eqref{eq:sys}. In this paper, we aim at describing some sufficient conditions of solvability of the LMI~\eqref{eq:lyap_lmi} when the decision variable $X$ satisfies additional sparsity constraints.

In order to simplify the presentation we say that a matrix $A\in\R^{N\times N }$ has \emph{$\alpha=\{k_1, \dots, k_n\}$-partitioning} with $N = \sum\limits_{i = 1}^n k_i$, if the matrix $A$ is written as follows
\[
A = \begin{pmatrix}
A_{1 1}    & A_{1 2}     & \dots   & A_{1 n} \\
A_{2 1}    & A_{2 2}     & \dots   & A_{2 n} \\
\vdots     & \vdots      & \ddots  & \vdots  \\
A_{n 1}    & A_{n 2}     & \dots   & A_{n n} 
\end{pmatrix}
\]
where $A_{i j}\in\R^{k_i\times k_j}$. We say that $A$ is \emph{$\alpha$-diagonal} if it is $\alpha$-partitioned and $A_{i j} = 0$ if $i\ne j$, and \emph{$\alpha$-lower triangular} if $A_{ i j} = 0$ if $i < j$. We aim at characterising \emph{$\alpha$-diagonally stable} matrices $A\in\R^{N\times N}$,  which are such that there exists an $\alpha$-diagonal positive definite $X\in\R^{N\times N}$ satisfying~\eqref{eq:lyap_lmi}. If $\alpha = \{1,\dots,1\}$, we say that an $\alpha$-diagonal (respectively, $\alpha$-lower triangular, $\alpha$-diagonally stable) matrix $A$ is \emph{diagonal} (respectively, lower-triangular, diagonally stable). 

We will make use of so-called \emph{scaled diagonally dominant} matrices.
\begin{defn}
A matrix $A \in\R^{n\times n}$ is called \emph{strictly row scaled diagonally dominant} if there exist positive scalars $d_1, \dots, d_n$ such that 
\[
d_i |a_{i i}| > \sum\limits_{j \ne i} d_j |a_{i j}|
\]
for all $i=1,\dots,n$. The matrix $A$ is \emph{strictly row diagonally dominant} if $d_i = 1$ for all $i$. 
\end{defn}

A related class to scaled diagonally dominant matrices is the class of $\mH$ matrices. In order to define this class we require the following definitions:
\begin{defn}[\cite{xiang1998weak}] \label{def:block-comp}
Given an $\alpha$-partitioned matrix $A$ with nonsingular $A_{ii}$ for all $i$, we define the \emph{$\alpha$-comparison matrix} $\cM^\alpha(A)$ as
\begin{equation}
\cM^\alpha_{ij}(A) = \left\{\begin{array}{ll} \|A_{ i i}^{-1}\|_2^{-1} &  \text{if }i = j, \\
-\|A_{i j}\|_2 & \textrm{otherwise},
\end{array} \right.\label{block-comparison}
\end{equation}
When $\alpha = \{1,\dots,1\}$, we will simply  write $\cM(A)$. 
\end{defn}

Note that $ \|A_{ i i}^{-1}\|_2^{-1} = \underline \sigma(A_{i i})$. Hence using a continuity argument we can assume that $\|A_{ i i}^{-1}\|_2^{-1} =0$ for a singular $A_{i i}$, and Definition~\ref{def:block-comp} is well-posed. 
The $\alpha$-partitioned matrices allow a version of \emph{Gershgorin circle theorem:
\begin{prop}[\cite{feingold1962block}]\label{prop:block-gershgorin}
For an $\alpha$-partitioned matrix $A\in\R^{N\times N}$, where $\alpha = \{k_1, \dots, k_n\}$ and $N = \sum\limits_{i = 1}^n k_i$, every eigenvalue of $A$ satisfies 
\begin{gather*}
\|(\lambda I - A_{ i i})^{-1}\|_2^{-1} \le \sum\limits_{j = 1, j\ne i}^n \|A_{i j}\|_2
\end{gather*}
for at least one $i$ where $i = 1, \dots, n$. 
\end{prop}
}

\begin{defn}
A matrix $A\in \R^{n\times n}$ is said to be  \emph{Metzler} if all the off-diagonal elements are positive.
\end{defn}
\begin{defn}
 A matrix $A\in \R^{n\times n}$ is said to be an $\mH$ matrix, if the minimal real part of the eigenvalues of $\cM(A)$ is greater than or equal to zero.
\end{defn}

It is clear that stable Metzler matrices are also $\mH$ matrices. It is also straightforward to show that $A$ is strictly row and column scaled diagonally dominant if and only if $\cM(A)$ has eigenvalues with positive real part~\cite{varga1976recurring}. We also refer the reader to~\cite{liu2004some}, \cite{hershkowitz1985lyapunov} for additional information on $\mH$ matrices.  

Let $\cDDp$ denote the cone of matrices $A$ such that $A$ and $A^T$ are strictly diagonally dominant and the elements on the diagonal of $A$ are positive (that is, $A_{i i} >0$). Similarly, let $\mHp$ denote $\cH$ matrices $A$ with positive elements on the diagonal of $A$.
If $A$ is a symmetric $\cDDp$ matrix, then by Proposition~\ref{prop:block-gershgorin} with $\alpha = \{1,\dots, 1\}$, 
it is easy to show that $A\succ 0$. Moreover, the constraint that $A = A^T \in \cDDp$ can be written as a set of linear constraints
\begin{equation}
\begin{aligned} \label{constraint-ddp}
&a_{i i} > \sum\limits_{j\neq i}^n c_{i j}~~\forall i, \\
& -c_{i j} \le a_{i j} \le c_{i j}, \textrm{ and } c_{i j} = c_{j i} ~~\forall i\ne j.
\end{aligned}
\end{equation}
Hence, a constraint $A\succ 0$ can be replaced by a more restrictive but scalable linear constraints. This approach was proposed in~\cite{dsos_ciss14,MajAT14} to restrict some sum-of-squares optimisation problems which are naturally SDPs to LPs. 

A symmetric $\mHp$ matrix is also positive semidefinite, and this constraint can be imposed by a number of second order cone constraints~\cite{boman2005factor}. That is $A=A^T\in \mHp$ if and only if
\begin{gather*}
A = \sum\limits_{i = 1}^N E_i^T X_i E_i, \text{  with  } X\in \cS^2_+,\, E_i \in \cT_2
\end{gather*}
where $E\in\cT_2$, if $E\in\R^{n\times 2}$ and every column of $E$ has only one non-zero element equal to one, and $N = |\cT_2|$. 

To summarise the subsection, we will mention this strong result on diagonal stability of $\mHp$ matrices, which we will revisit in the sequel. 

\begin{prop}[\cite{hershkowitz1985lyapunov}]\label{prop:known-res}
Let $-A$ be an $\mHp$ matrix. Then $A$ is diagonally stable if and only if $A$ is nonsingular. 
\end{prop}



\section{Motivation and Problem Formulation}\label{s:motivation}

\subsection{Conic Programming}
Conic optimisation problems take the generic form of optimising a linear functional over the intersection of an affine subspace and a proper cone. Typically conic programmes have the following primal and dual formulations:
\begin{align*}
\text{minimise}  &\quad c^Tx  &\text{maximise} & \quad b^T y\\
\text{s.t.} \quad &Ax = b     &\text{s.t.} \quad & c - A^T y =s \\
&x \in \cK \quad & & (y,s)\in (\R^m, \cK^{\ast})
\end{align*}
where $\cK$ is a \emph{proper} cone (i.e. closed, non-empty, pointed, convex) and $\cK^{\ast}$ is the dual cone of $\cK$ defined as 
\begin{equation*}
\cK^{\ast}:= \left\{y~|~ \langle y,x \rangle \ge 0, \forall x \in \cK \right\}.
\end{equation*}
It is well known, that the cone of positive semidefinite matrices $\cS_+^n$ is self dual meaning that $(\cS_+^n)^\ast = \cS_+^n$. The cone of symmetric $\cDDp$ matrices
\begin{gather*}
\cK_{\rm LP} = \left\{X \in\cS^n\cap \cDDp \right\},
\end{gather*}
however, is not self-dual and it is larger than the cone $\cS_+^n$. More specifically:
\begin{gather*}
\cK_{\rm LP}^\ast = \left\{ X\in \cS^n \bigl| v_i^T X v_i\ge 0, ~\forall v_i\in \cT_1 \right\},
\end{gather*}
where $\cT_1$ is the set of all vectors in $\R^n$ with a maximum of two non-zero elements, each of which is $\pm1$. The cone of symmetric $\mHp$ matrices defined as
\begin{gather*}
\cK_{\rm SOCP} = \left\{X \in\cS^n\cap \mHp \right\},
\end{gather*}
has the dual 
\begin{gather*}
\cK_{\rm SOCP}^\ast = \left\{X \in\cS^n\bigl| E_i^T X E_i \succeq 0,~\forall E_i \in \cT_2  \right\}.
\end{gather*}
We will make use of these cones and their duals in the remainder of the paper.

\subsection{Structured Gramians via Basis Pursuit}
The standard from primal SDP~\cite{BoyGFBV94} is written as
\begin{align}
\min_X \quad& \langle C,X\rangle \label{eq:SDPprimal}\\
\text{s.t} \quad&X\in \mathcal{S}_n^+, \quad  \langle A_i,X\rangle = b_i, \quad i = 1,\hdots,m \nonumber
\end{align} 
where $\mathcal{S}_n^+$ is the cone of $n\times n$ positive semidefinite matrices. The basis pursuit algorithm proceeds as follows: At each iteration the algorithm re-parameterizes the simpler cone that approximates $\mathcal{S}_n^+$ and then solves an optimization problem over this cone, the solution of which is then used to update the cone for the next iteration. 
In particular, the algorithm specifies for a fixed matrix $L$, the cone
\begin{equation*}
\mathcal{K}(L) = \left\{ X  ~|~ X = L^TQL,~Q = Q^T \in\cDDp \right\}.
\end{equation*}
Note that $Z\in \mathcal{K}(L) \Rightarrow Z \succeq 0$. The algorithm in \cite{basis_pursuit}  solves a sequence of optimization problems of the form \eqref{eq:SDPprimal} but with the conic constraint replaced by $X\in \mathcal{K}(L_k)$ where the sequence $\left\{L_k \right\}$ is given by
\begin{align*}
L_0 &= I\\
L_{k+1} &= \text{decomp}(X_k)
\end{align*}
where $\text{decomp}(X_k)$ is a Cholesky decomposition of $X_k$, the optimal solution decision variable from iteration $k$. In some cases, $X_k$ can have singular values close to zero, thus creating numerical problems in the iterative procedure. In order to avoid such cases, we can remove a $k$-th column of $L$ with the $k$-th entry close to zero (we assume that $L$ is lower triangular). One can also use LDL decomposition to avoid dealing with negative eigenvalues of $X_k$. This approach also slightly improves numerical complexity of the conic programme by lowering the number of decision variables and constraints. We finally note that the method relies on the fact that at the first iteration a feasible solution $X_0$ exists. 

In many applications, it is desirable to solve the following problem using an LP or SOCP rather than the more natural SDP:
\begin{equation}\label{main-lmi}
\begin{aligned}
 \min ~&\tr(X) \\
\text{such that:}~& A X + X A^T \prec 0 \\
                   & X=X^T \text{ is $\alpha$-diagonal},
\end{aligned}
\end{equation}
where $A$ is Hurwitz, and $\alpha$ is a given partitioning. Note that since $A$ is Hurwitz, then the condition $X\succ 0$ is implied by the solvability of~\eqref{main-lmi}. The basis pursuit can be applied given an $X$ satisfying the constraints of the programme~\eqref{main-lmi}. Therefore, we set up our problem: \emph{find $X$ satisfying the constraints of~\eqref{main-lmi} with algebraic or linear programming methods}.

There are many practical applications, where the problem of the form~\eqref{main-lmi} appears and one of them is structured model reduction. For example, consider the boiler-header system described in~\cite{trnka2013structured} 
and schematically depicted in Figure~\ref{fig:boiler}. 
The state space can be partitioned according to dimensions of the subsystems, which are $\{3,~3,~1\}$. The system always admits diagonal generalised Gramians, since the drift matrix of the closed loop system is a stable $\mHp$ matrix. 
\begin{figure}[t]
  \centering
\includegraphics[width = 0.5\columnwidth]{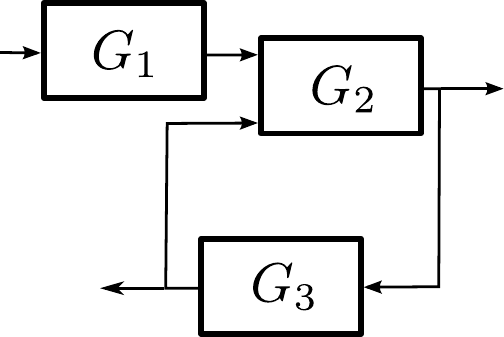}
  \caption{Block-diagram of the boiler-header system.}
  \label{fig:boiler}
\end{figure}

In order to perform structured model reduction~\cite{trnka2013structured} the authors computed a $\{3,3,1\}$-diagonal generalised controllability Gramian $P = \diag{P_1, P_2, P_3}$ such that $P_1, P_2 \in\R^{3\times 3}$ and $P_3 \in \R$. The optimal trace of such a Gramian computed using semidefinite programming is equal to $1.2817\cdot 10^4$. Using linear programming, we found the minimum trace of $2.4132\cdot 10^4$ signifying a loss of quality of almost $100\%$. After solving one iteration of the basis pursuit algorithm we obtain objective equal to $1.5893 \cdot 10^4$, an additional iteration of the algorithm gives $1.3172 \cdot 10^4$, and one more provides a value equal to $1.3093 \cdot 10^4$, which comes really close to the optimal value. Naturally, on this example we do not need basis pursuit or linear programming to obtain an optimal solution due to the low complexity of the problem. However, this example indicates that the basis pursuit algorithm can be beneficial to obtain an approximate solution of large scale Lyapunov inequalities using linear programmes. 

\section{$\mH$ Matrices and Diagonal Stability} \label{s:sdd}
The main result of this section concerns diagonal stability of $\mH$ matrices, where we sharpen the results from~\cite{hershkowitz1985lyapunov} by providing an explicit diagonal Lyapunov function for a class of $\mHp$ matrices.
\begin{thm} \label{prop:h-mat-lyap}
Let $-A$ be an $\mHp$ matrix with a nonsingular $\cM(A)$. Then the following conditions hold
\begin{enumerate}
\item There exist positive vectors $v = \begin{pmatrix} v_1 & \dots & v_n \end{pmatrix}^T$, $w = \begin{pmatrix} w_1 & \dots & w_n \end{pmatrix}^T$ such that $\cM(A) v$, $w^T \cM(A)$ are also positive.
\item There exists a diagonal $X$ such that $- (A X + X A^T)$ is an $\mHp$ matrix. Moreover, we can choose it as $X = P_v P_w^{-1}$, where $P_v =\diag{v_1, \dots, v_n}$, $P_w =\diag{w_1, \dots, w_n}$, and $v$, $w$ satisfy point 1).
\item There exists a diagonal positive definite matrix $Y$ such that  
\begin{gather}\label{con:linear}
-P_w A P_w^{-1} Y  -  Y P_w^{-1} A^T P_w\in \cDDp
\end{gather}
\end{enumerate}
\end{thm}
\begin{proof} 
1) By definition $-\cM(A)$ is a Metzler matrix with all eigenvalues $\lambda_i(\cM(A)) \le 0$, since $\cM(A)$ is nonsingular by the premise, $-\cM(A)$ is a Hurwitz Metzler matrix. Hence the claim follows by applying the results from~\cite{rantzer2015ejc}.

2) Let $X = P_v P_w^{-1}$, then 
\begin{multline*}
(\cM(A) X + X \cM(A^T)) w = (\cM(A) v + X  \cM(A)^T w ) \gg 0,
\end{multline*}
where the inequality follows since $\cM(A) v$ and $\cM(A)^T w$ are positive and $X$ is nonnegative. Hence $S = -\cM(A) X - X \cM(A^T)$ is a Metzler matrix and there exists a positive vector such that $S w$ is negative. This implies that $S$ is a symmetric Hurwitz and Metzler matrix, which means that $\cM(A) X + X \cM(A^T)$ is positive definite.

Note that $a_{i i} < 0$ for all $i$, let
\begin{gather*}
(-A X - X A^T)_{i j} = - a_{i j} x_j - a_{j i} x_i \\
(\cM(A) X + X \cM(A^T))_{i j} = 
\begin{cases} - a_{i j} x_j - a_{j i} x_i  & i = j \\
- |a_{i j}| x_j - |a_{j i}| x_i            & i \ne j
\end{cases}
\end{gather*}

It is straightforward to show that $\cM(A) X + X \cM(A^T) \le \cM(- A X - X A^T)$, moreover the elements on the diagonal are equal. This means that we can write $\cM(A) X + X \cM(A^T) = s I - R_1$, $\cM(- A X - X A^T) = s I -R_2$, where the matrices $R_1$ and $R_2$ satisfy $R_1\ge R_2 \ge 0$. According to Weilandt's theorem  $\rho(R_1) \ge \rho(R_2)$ (cf.~\cite{hershkowitz1985lyapunov}). Therefore the minimal eigenvalue of $\cM(A) X + X \cM(A^T)$ is smaller or equal to the minimal eigenvalue of $\cM(- A X - X A^T)$. This implies that $\cM(- A X - X A^T)$ has eigenvalues with positive real part, hence $- A X - X A^T$ is an $\mHp$ matrix.

3) Consider the matrix $R = P_w \cM(A) P_v$, and $e$ the vector of ones. Now it is easy to see that $R e \gg 0$:
\begin{gather*}
P_w \cM(A) P_v e = P_w \cM(A) v  \gg 0.
\end{gather*}
This implies that the matrix $P_w \cM(A) P_v$ is row strictly diagonally dominant. Similarly, we can show that $P_w \cM(A) P_v$ is column strictly diagonally dominant. This by definition implies that the matrix $-P_w A P_v$ is a row and column diagonally dominant matrix with positive elements on the diagonal or a $\cDDp$ matrix. Hence the matrix $-P_w A P_v - P_v A^T P_w$ is positive definite. Since we can set $Y=P_v P_w$ the result follows.
\end{proof}

We showed that there exists a diagonal $X$ matrix such that the matrix $Z = -P_w A P_w^{-1} X - X P_w^{-1} A^T P_w$ is a $\cDDp$ matrix and hence positive definite. Note that the constraint $Z =Z^T \in\cDDp$ is linear and if needed we can relax the sparsity constraints on $X$. This implies that given an $\mH$ drift matrix, we can compute an $\alpha$-diagonal Lyapunov function with an \emph{arbitrary $\alpha$} using \emph{linear programming}.

If the entries of the $A$ matrix are poorly scaled then solving a linear programme can be numerically challenging. Using our methods, this can be avoided if we compute an initial point using the right and left eigenvectors of $\cM(A)$, instead of the positive vectors $v$ and $w$ satisfying point 1). Having an initial point re-scales the optimisation programme and can provide feasible points as shown on a specific example in~\cite{sootla2015stoch}.

Theorem~\ref{prop:h-mat-lyap} is a direct generalisation of the similar result for Metzer matrices (cf.~\cite{rantzer2015ejc}), but our result can be applied to a broader class of matrices including lower-triangular matrices. Using Theorem~\ref{prop:h-mat-lyap} other results for Metlzer matrices can be extended to problems such as construction of sum- and max-separable Lyapunov functions (cf.~\cite{rantzer2015ejc}).

The state-space transformation $P_w$ is essential in order to guarantee the diagonal dominance of the inequality. Consider an asymptotically stable matrix 
\begin{gather*}
A = \begin{pmatrix}
 -1 & -2 \\
 2  & -5
\end{pmatrix}.
\end{gather*} 
and a positive definite $X = \diag{x_1, x_2}$. The matrix $-A$ is an $\cH_+$ matrix and it is stable. The diagonal dominance of $A X + X A^T$ requires the following inequalities to be fulfilled
\begin{gather*}
2\cdot x_1 > 2 x_2 +  2 x_1, \qquad 2\cdot 5 x_2 > 2 x_2 + 2 x_1 
\end{gather*} 
for some positive $x_1$, $x_2$. The first inequality is equivalent to $0 > 2 x_2$, which is impossible to fulfil.


\section{$\alpha$-Diagonal Stability and $\mHp$ Matrices\label{s:bsdd}}

\subsection{A Motivating Example}
In this section, we cover two main classes of results for diagonal stability and compare them to a classical example from~\cite{arcak2006diagonal} for cyclic systems. These classes stem from two arguments based on the  \emph{passivity}  and the \emph{small gain theorem}. In this section, we will argue that the $\mHp$ matrix condition is an implicit constraint  in these stability proofs. In order to explain our motivation consider an example studied in~\cite{arcak2006diagonal} and let:
\[
A_n^0 = \begin{pmatrix}
0_{1 \times n-1}   &-\beta_1\\
\diag{\beta_2,\dots, \beta_n} & 0_{n-1\times 1}
\end{pmatrix} -\diag{\alpha_1, \dots, \alpha_n}
\]
where $\alpha_i$, $\beta_i$ are positive scalars. This matrix represents the dynamics of a negative feedback of a cascade of transfer functions $G_i(s) = \frac{\beta_i}{s + \alpha_i}$.  First, let us consider the $2$ by $2$ case, which gives 
\[
A_{2}^0 = \begin{pmatrix}
 -\alpha_1&  -\beta_1     \\
  \beta_2 & -\alpha_2   
\end{pmatrix}
\]
and two transfer functions $G_1 =\frac{ \beta_1}{(s + \alpha_1)}$ and $G_2 = \frac{\beta_2}{(s+\alpha_2)}$. According to the small gain theorem, the system is stable if
\[
\left\| \frac{\beta_1}{s + \alpha_1} \right\|_\Hinf  \left\| \frac{\beta_2}{s + \alpha_2} \right\|_\Hinf  =  \frac{\beta_1 \beta_2}{\alpha_1 \alpha_2} < 1.
\]
This argument can be extended to an arbitrary size matrix resulting in the condition 
\begin{equation}\label{cond-h-mat}
\frac{\beta_1 \cdots \beta_n }{\alpha_1 \cdots \alpha_n} < 1.
\end{equation}

Surprisingly, it is straightforward to verify by definition that $A_n^0$ is an $\mHp$ matrix if and only if~\eqref{cond-h-mat} holds. Hence on this loop \emph{the $\mHp$ matrix condition is a small gain condition}. Alternatively, using passivity arguments it was shown in~\cite{arcak2006diagonal}, that $A_n^0$ is asymptotically stable if and only if

\begin{equation}\label{cond-diag-stab}
\frac{\beta_1  \cdots \beta_n }{\alpha_1 \cdots \alpha_n} < (\sec(\pi/n))^n.
\end{equation}

This in particular means that for $n =2$ all matrices in the form $A_2^0$ are not only stable, but also diagonally stable, however, they may not be $\mH$ matrices. This analysis is based on passivity arguments and has been extended to less restrictive classes of systems in~\cite{arcak2011diagonal}.  It is easy to verify that with $n\rightarrow\infty$  the limit $(\sec(\pi/n))^n$ converges to one. Hence, it appears (for this class of system) that for large dimensions, $\mH$ matrices constitute a large subset of diagonally stable matrices.  We will pursue the relation between $\mHp$ matrices and small gain argument in the $\alpha$-diagonal case in the remainder of the paper. 

\subsection{Passivity and Small Gain Conditions for $\alpha$-Diagonal Stability}
\begin{figure}[t]
\centering
\includegraphics[width=0.5\columnwidth]{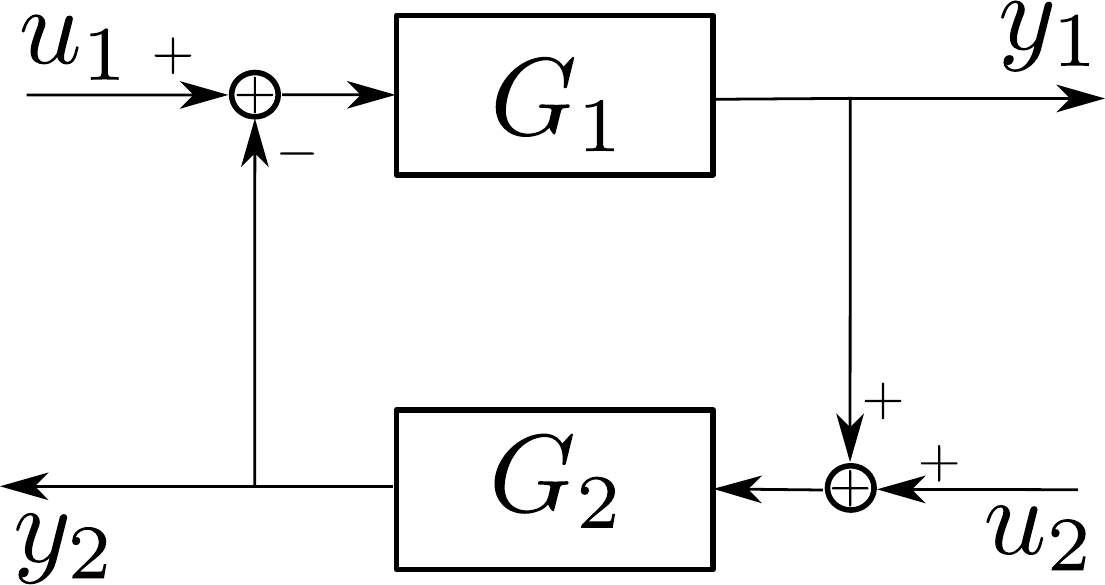}
\caption{Feedback interconnection of two stable systems $G_1$, $G_2$.} \label{fig:feedback}
\end{figure}

Let $G^c$ be the closed loop transfer function depicted in Figure~\ref{fig:feedback}, which is an interconnection of two Linear Time Invariant (LTI) subsystems 
\[
G_i = \left[\begin{array}{c|c}
A_i & B_i \\
\hline
C_i & D_i
\end{array}\right],
\]
where $A_i\in \R^{k_i \times k_i}$, $B_i\in \R^{k_i \times m_i}$, $C_i\in \R^{l_i \times k_i}$, $D_i\in \R^{l_i \times m_i}$ and $m_1 = l_2$, $m_2 = l_1$. The closed loop transfer function from $[u_1,u_2]$ to $[y_1,y_2]$ has the following state-space realisation
\[
G^c = \left[\begin{array}{cc|cc}
A^c_{1 1} & A_{1 2}^c & B_{1 1}^c & B_{1 2}^c\\
A^c_{2 1} & A_{2 2}^c & B_{2 1}^c & B_{2 2}^c\\
\hline
C^c_{1 1} & C_{1 2}^c & D_{1 1}^c & D_{1 2}^c\\
C^c_{2 1} & C_{2 2}^c & D_{2 1}^c & D_{2 2}^c
\end{array}\right],
\]
where 
\begin{align}
A_{1 1}^c &= A_{1} - B_1 R_{21} D_2 C_1, \quad &A_{1 2}^c &=- B_1 R_{21} C_2  \label{eq:partition}\\
A_{2 1}^c &=B_2 R_{12} C_1, \quad  &A_{2 2}^c &= A_{2}- B_2 R_{12} D_1 C_2 \nonumber
\end{align}
and  $R_{12} = (I + D_1 D_2)^{-1}$, $R_{21} = (I + D_2 D_1)^{-1}$ and the rest of the matrices are computed accordingly. For the sake of simplicity we assume that this realisation is minimal.

Passivity and small gain arguments can both be used to determine if the closed loop system is stable but we will focus on the small gain condition. Passivity results in this direction will be addressed in future work, similar ideas were pursued in~\cite{sturk2011structured,AndTSP11}.

It is straightforward to verify that stability of the system with inputs $u_1$, $u_2$ and outputs $y_1$, $y_2$ depends on stability of the transfer function 
$L = (I - G_2 G_1)^{-1}$. 
\begin{prop}[Small Gain Theorem]\label{thm:sg}
Suppose $\cB$ is a Banach-algebra and $Q\in \cB$. If $\|Q\|<1$, then $(I-Q)^{-1}$ exists and 
\begin{equation*}
(I-Q)^{-1}=\sum_{k=0}^\infty Q^k.
\end{equation*} 
\end{prop}
Applying Proposition~\ref{thm:sg} we can verify that if $Q:=\|G_2 G_1\|_\Hinf \le \|G_2\|_\Hinf \|G_1\|_\Hinf < 1$ then the function $L$ and hence the closed loop are stable  (cf.~\cite{KhalilNonlinearControl}). We can apply the small gain condition to the closed transfer function, which would result in a condition on $\alpha$-diagonal stability of the matrix $A^c$. However, given only a partitioning $\alpha = \{k_1, k_2\}$ and a realisation of the closed loop transfer function $G^c$, these conditions again will be hard to verify. We can apply a small gain theorem in another way, namely apply it to the matrix $A^c$ directly. In this case, we do not need to know the realisation of transfer functions $G_1$, and $G_2$, all we need to know is the matrix $A^c$ and the partitioning $\alpha$. The conditions on $\alpha$-diagonal stability of $A^c$ are established in the following proposition.

\begin{prop}\label{prop:small-gain} Let $A^c$ be $\alpha$ partitioned with $\alpha = \{k_1, k_2\}$ 
\begin{gather*}
A^c = \begin{pmatrix}
A_{1 1}^c & A_{1 2}^c \\
A_{2 1}^c & A_{2 2}^c
\end{pmatrix}.
\end{gather*}
Let $K_1(s) =  -(s I - A_{1 1}^c)^{-1} A_{1 2}^c$, $K_2(s) = (s I - A_{2 2}^c)^{-1} A_{2 1}^c$ with Hurwitz $A_{1 1}^c$, $A_{2 2}^c$. If there exists a $\gamma>0$ such that 
$\|K_1 \|_\Hinf < 1/\gamma$ and $\|K_2 \|_\Hinf < \gamma$, then the matrix $A^c$ is $\alpha$-diagonally stable.  
\end{prop}
\begin{proof} 
We need to show that there exists an $\alpha$-diagonal Lyapunov function for the system $\dot{x}^c=A^c x^c$. For the sake of clarity we drop the superscript $c$ from $A_{i j}^c$ and simply write $A_{i j}$. The inequality $\|K_1 \|_\Hinf < 1/\gamma$ and the Bounded Real Lemma imply that  $X_1 \succ 0$ solves the  Riccati equation
\begin{gather}\label{ineq:brlg1}
0  = A_{1 1} X_1 +  X_1 A_{1 1}^T + \gamma^2 X_1  X_1 + A_{1 2} A_{1 2}^T  = \\
\notag A_{1 1} X_1 +  X_1 A_{1 1}^T + \begin{pmatrix}
  X_1  & A_{1 2}
 \end{pmatrix} \begin{pmatrix}
 \gamma^2 I_{k_1} & 0 \\ 0 & I_{k_2} 
 \end{pmatrix}\begin{pmatrix}
   X_1  \\ A_{1 2}^T 
 \end{pmatrix}, 
\end{gather}
which has always has a solution since $(I, A_{1 1})$ is a controllable pair (cf.~\cite{ZDG}), since the control matrix is equal to $I$ and we can control every state independently.

Again, due to the Bounded Real Lemma the inequality $\|K_2 \|_\Hinf < \mu$ is equivalent to 
\begin{gather}\label{eq:br2}
 A_{2 2} Y_2  + Y_2 A_{2 2}^T +   \mu^{-2} Y_2   Y_2 + A_{2 1} A_{2 1}^T = 0,
\end{gather}
where  $Y_2\succ 0$ since $(I, A_{2 2})$ is a controllable pair (cf.~\cite{ZDG}). Let $\mu = \gamma - \varepsilon$ for some $\varepsilon>0$ such that $\mu > \|K_2\|_\Hinf$, which implies that  $\mu^{-2}  Y_2   Y_2 \succ \gamma^{-2} Y_2 Y_2$ and consequently:
\begin{gather*}
 A_{2 2} Y_2  + Y_2 A_{2 2}^T +  Y_2   \gamma^{-2}  Y_2 + A_{2 1} A_{2 1}^T \prec 0
\end{gather*}

By multiplying the equation by $\gamma^{-2}$ setting $X_2 = Y_2 \gamma^{-2}$
 \begin{multline}
 \label{ineq:brlg2}
 A_{2 2} X_2 + X_2 A_{2 2}^T +   X_2  X_2 \\
+ \gamma^{-2} A_{2 1} A_{2 1}^T \prec 0 \Leftrightarrow\begin{pmatrix}
 \gamma^2 I_{k_1} & 0 \\ 0 & I_{k_2} 
 \end{pmatrix}  \\
 +\begin{pmatrix}
   A_{2 1}^T  \\ X_2
 \end{pmatrix} (A_{2 2}^T X_2 +  X_2 A_{2 2})^{-1} \begin{pmatrix}
  A_{2 1}  & X_2
 \end{pmatrix} \succ 0
\end{multline}

Combining the inequalities \eqref{ineq:brlg1} and~\eqref{ineq:brlg2} yields 
\begin{multline}\label{ineq:blkdiag}
0 \succ A_{1 1} X_1 +  X_1 A_{1 1}^T - \begin{pmatrix}
  X_1  & A_{1 2}
 \end{pmatrix} \begin{pmatrix}
  A_{2 1}^T   \\ X_2
 \end{pmatrix}\\ 
 \cdot ( A_{2 2} X_2 +  X_2 A_{2 2}^T)^{-1}
  \begin{pmatrix}
   A_{2 1} & X_2
 \end{pmatrix} \begin{pmatrix}
   X_1  \\ A_{1 2}
 \end{pmatrix}    \\
=A_{1 1} X_1 +  X_1 A_{1 1}^T - (X_1 A_{2 1}^T + A_{1 2} X_2 ) 
   \\ \cdot(X_2 A_{2 2} + A_{2 2}^T X_2)^{-1} (A_{2 1} X_1  + X_2 A_{1 2}^T ).
\end{multline}
Applying the Schur complement properties to \eqref{ineq:blkdiag} yields
\[
\begin{pmatrix}
A_{1 1} X_1  + X_1 A_{1 1}^T  & A_{1 2} X_2 + X_1 A_{2 1}^T \\
A_{2 1} X_1  + X_2 A_{1 2}^T  & A_{2 2} X_2 + X_2 A_{2 2}^T
  \end{pmatrix} \prec 0,
\]
thus the blocks on the diagonal are negative definite which completes the proof.
\end{proof}

Our proof is constructive, \emph{and} shows how to build an $\alpha$-diagonal Lyapunov function by solving two Riccati equations~\eqref{ineq:brlg1} and~\eqref{eq:br2} instead of solving an LMI. Next we link a simplified version of these conditions with $\alpha$-partitioned and $\mHp$ matrices. 

\subsection{Conditions for $\alpha$-Diagonal Stability via $\mHp$ Matrices}
The authors in~\cite{feingold1962block} showed that $A$ is Hurwitz if it is an $\alpha$-partitioned matrix such that $\cM^\alpha(A) \in \cDDp$, and the matrices $A_{i i}$ are Hurwitz and Metzler for all $i$. In particular, this result shows that stability of $A$ is implied by stability of all the blocks $A_{i i}$. We provide a generalisation of this result.

\begin{lem}\label{prop:alpha-stability}
Let $A$ be $\alpha$-partitioned matrix and $\cM^\alpha(A)$ be an $\mHp$ matrix. Let also $A_{i i}$ be Hurwitz matrices,  and the Hamiltonian matrices 
\begin{equation}
H_i = \begin{pmatrix}
A_{i i} & \gamma_{i}^{-2} I \\
-I & - A_{i i}^T
\end{pmatrix} \label{hinf-cond}
\end{equation}
have no purely imaginary eigenvalues with $\gamma_i = \|A_{i i}^{-1}\|_2 + \varepsilon$ for all $\varepsilon >0$. Then $A$ is a Hurwitz matrix.
\end{lem}
\begin{proof}
We prove the result by contradiction. Let $A$ have eigenvalues with a positive real part. Since $\cM^\alpha(A)$ is an $\mHp$ matrix, there exists positive scalars $d_i$ such that for every $i$
\begin{gather} \label{cond:h-mat}
\|A_{i i}^{-1}\|_2^{-1} >\sum\limits_{i\ne j}  \|A_{i j}\|_2 \frac{d_j}{d_i}.
\end{gather}

The matrix $A$ is unstable if and only if $D^{-1} A D$ is unstable with $D = \diag{d_1 I_{k_1}, \dots , d_n I_{k_n}}$. Let $\lambda$ be the eigenvalue of $D^{-1} A D$ with a positive real part. By Proposition~\ref{prop:block-gershgorin} 
there exists an index $i$ such that 
\begin{gather} \label{block-gershgorin}
\|(\lambda I - A_{i i})^{-1}\|_2^{-1} \le \sum\limits_{i\ne j} \left\|A_{i j}\frac{d_j}{d_i}\right\|_2 = \sum\limits_{i\ne j} \|A_{i j}\|_2 \frac{d_j}{d_i}.
\end{gather}

Now since the Hamiltonian matrix $H_i$ has no purely imaginary eigenvalues for all $\varepsilon > 0$ and $A_{i i}$ is Hurwitz, this implies that $\|(sI - A_{ i i})^{-1}\|_\Hinf =\|A_{ i i}^{-1}\|_2$. Therefore the maximum 
of $\|(z I - A_{ i i})^{-1}\|_2$ over $z$ with $\Re(z)\ge 0$ is equal to $\|A_{ i i}^{-1}\|_2$, and $\|(\lambda I - A_{ i i})^{-1}\|_2 \le \|A_{ i i}^{-1}\|_2$. 
Hence due to~\eqref{cond:h-mat}
\begin{gather*}
\|(\lambda I - A_{i i})^{-1}\|_2^{-1}\ge \|A_{i i}^{-1}\|_2^{-1} >\sum\limits_{i\ne j}  \|A_{i j}\|_2 \frac{d_j}{d_i}.
\end{gather*}
We arrive at the contradiction with~\eqref{block-gershgorin}, which completes the proof.
 \end{proof}

Lemma \ref{prop:alpha-stability} allows us to determine stability of $A$ by verifying stability of the blocks $A_{ii}$ subject to the condition~\eqref{hinf-cond} and $\cM^\alpha(A)$ being an $\mHp$ matrix. This, however, does not directly imply that there exists an $\alpha$-diagonal Lyapunov function.  In what follows, we only present the result for $\alpha =  \{k_1, k_2\}$ partitioning.

\begin{thm} Let $A$ be $\alpha$ partitioned with $\alpha = \{k_1, k_2\}$, then under the premise of Lemma~\ref{prop:alpha-stability} the matrix $A$ is $\alpha$-diagonally stable. 
\end{thm}
\begin{proof} 
The proof is using the small gain argument for the systems $G_1(s) =  (s I - A_{1 1})^{-1} A_{1 2}$, $G_2(s) = (s I - A_{2 2})^{-1} A_{2 1}$. 
We have that $\|G_1\|_\Hinf \|G_2\|_\Hinf \le \Delta$ where
\begin{gather*}
 \Delta := \|A_{2 1}\|_2 \|(sI - A_{1 1})^{-1}\|_\Hinf \|A_{1 2}\|_2 \|(sI - A_{2 2})^{-1}\|_\Hinf .
\end{gather*}
Under the premise of Lemma~\ref{prop:alpha-stability} we have that $\gamma \|A_{1 2}\|_2 < \| A_{1 1}^{-1}\|_2^{-1}$ and  $\gamma^{-1}\|A_{2 1}\|_2 < \| A_{2 2}^{-1}\|_2^{-1}$. Hence 
$\|G_1 \|_\Hinf < \gamma^{-1}$, while $\|G_2 \|_\Hinf < \gamma$.  Proposition~\ref{prop:small-gain} proves the claim.
\end{proof}

Note that if $A$ is such that $\cM^\alpha(A)$, $\cM^\alpha(A^T)\in \cDDp$, it is not generally true that $\cM^\alpha(A+A^T)\in\cDDp$. This property holds for $\alpha =\{1,\dots, 1\}$ and was used in the proof of Theorem~\ref{prop:h-mat-lyap}. Hence the absence of this property for a general $\alpha$ is the major obstacle for extending Theorem~\ref{prop:h-mat-lyap} to the $\alpha$-diagonal case.
\section{Numerical Example}\label{sec:example}
Consider the one-dimensional heat equation in the form 
\begin{align*}
&\frac{\partial T(t,x)}{\partial t} = \alpha \frac{\partial^2 T(t,x)}{ \partial x^2} + u(x,t) && x\in(0,1), t >0 \\
&T(0,t)  = T(1,t) = 0, 																		     && t\ge 0 \\
&T(x, 0) = 0                                                                                  && x\in[0,1]
\end{align*}
with $\alpha = -0.01$, where $T(t,x)$ denotes the temperature at time $t$ at $x$. Assume, we want to heat (i.e. apply an input) at a point of the rod located at $1/3$ of its length across, and observe the temperature at a point on the rod located at $2/3$ of its length. Then as in~\cite{MRBench}, we can obtain the following spatially discretised model: 
\begin{align*}
\dot{X}(t) &= AX(t) + Bu(t),\quad X(0)=0,\\
Y(t) &= CX(t), 
\end{align*}
where $X(t)\in \R^n$ is the temperature at time $t$ at each of the $n$ spatial discretisation points, and   
\begin{gather}
A = \alpha (n+1)^2
\begin{pmatrix}
2  & -1    &        &        &    \\
-1 &  2    & -1     &        &    \\
   &\ddots & \ddots & \ddots &    \\
   &       & -1     & 2      & -1 \\
   &       &        & -1     &  2     
\end{pmatrix}
\in \R^{n\times n}, 
\end{gather}
The matrices $B \in \R^{n\times 1}$, $C\in \R^{1\times n}$ are equal to zero except for the entries $\lceil n/3 \rceil $ and $\lceil 2 n/3 \rceil$, respectively, which are equal to one.

\begin{table}
\caption{Time to compute the generalised controllability Grammian} \label{tab:time-res}
\centering
\begin{tabular}{c c c c c}
Size of the system        &   $50$  & $100$  & $150$  & $200$ \\
\hline
\hline
SDP solution              &  $0.94$ & $22.7$ &$310.7$ & NA  \\
SOCP relaxation &  $0.74$ & $4.11$ & $11.9$ & $31.2$ \\
LP relaxation  &  $0.01$ & $0.02$ & $0.05$ & $0.10$ \\ 
LP relaxation w scaling &  $0.01$ & $0.03$ & $0.05$ & $0.10$ \\
\hline
\end{tabular}
\end{table}

Our goal is to compute the diagonal controllability  Gramians $P$ for various $n$ with a minimal trace, which we will do in the dual form:
\begin{align*}
\max\limits_Y~~ & \tr(B B^T Y), \\
\textrm{s.t.~~} & \diag{Y A + A^T Y + I} = 0 \\
                & Y \prec 0.
\end{align*}

In the dual form, we have an LP relaxation where $-Y$ belongs to the dual to the cone of symmetric $\cDDp$ matrices, and an SOCP relaxation, $-Y$ belongs to the dual to the cone of symmetric $\mHp$ matrices. We solve only the dual SDP formulation and the corresponding relaxation.  Due to the structure of the system, the trace of its Gramians does not change much with dimensions and we always get the optimal values in the range between $6.5$ to $6.6$ for the SDP programme. Remarkably the results for the SOCP relaxation are only slightly higher, but in the same range of values. This however, is due to structure of the system, where the drift is Metzler and the matrix $B B^T$ has only one non-zero entry on the diagonal. The optimal solutions for the LP relaxation are in the range between $10.7-10.9$, hence there is a drop in quality when using this relaxation. 

In Table~\ref{tab:time-res}, we provide the computational times for various systems sizes $n$. The entry ``NA'' means that the programme terminated due to running out of memory. Note however, that we do not take into account the time for parsing the constraints (that is, we plot only the variable ``solvertime'' in Yalmip~\cite{YALMIP}). Since $A$ is a Metlzer matrix it is straightforward to find a transformation $T$ such that $T A T^{-1}$ becomes a diagonally dominant matrix (see Theorem~\ref{prop:h-mat-lyap}). We have implemented the LP relaxation while transforming the $A$, $B$ matrices with such a transformation $T$. The optimal solutions for the trace vary between $7.1$ and $7.3$, thus drastically improving the quality of the relaxation with a mild loss (if any) in computational time.

\section{Discussion and Conclusion\label{s:con}}
We have provided some sufficient conditions on $A$, which guarantee the existence of feasible points in~\eqref{main-lmi} and interpreted these results as small gain like conditions. Moreover, our sufficient conditions also provide computationally cheap solutions, for example Proposition~\ref{prop:small-gain} replaces an LMI constraint with two Riccati Equation solutions. If we drop the ``$X$ is $\alpha$-diagonal'' constraint and set $Q = A X + X A^T$, then the LMI~\eqref{eq:lyap_lmi} has a solution for any $Q\prec 0$ \emph{if and only if} $\lambda_i(A)+\bar{\lambda}_j(A)\neq 0$. Since $Q$ is arbitrarily negative definite, we can replace the constraint $A X + X A^T\prec 0$ with $-A X - X A^T\in\cDDp$. Thus our solvability LMI becomes a linear program. Finally we showed how our constructive proofs can be used to initiate a recently developed basis pursuit algorithm for solving large scale optimization problems.

\bibliography{bibl_hmatr}
\end{document}